\documentclass{amsart}
\usepackage{amssymb,latexsym}
\input epsf
\usepackage{graphicx}
\newtheorem{theorem}{Theorem}[section]
\newtheorem{proposition}[theorem]{Proposition}
\newtheorem{corollary}[theorem]{Corollary}
\newtheorem{definition}[theorem]{Definition}
\newtheorem{remark}[theorem]{Remark}
\newtheorem{lemma}[theorem]{Lemma}
\newtheorem{example}[theorem]{Example}

\newcommand{\lk}{{\rm link}}

\newcommand{\cC}{{\mathcal C}}

\newcommand{\iI}{{\mathcal I}}
\newcommand{\sS}{{\mathcal S}}

\newcommand{\ZZ}{{\mathbb Z}}

\newcommand{\sm}{{\smallsetminus}}
\begin{document}
\title[The absolute order on the symmetric group is Cohen-Macaulay]
{The absolute order on the symmetric group, constructible
partially ordered sets and Cohen-Macaulay complexes}

\author{Christos~A.~Athanasiadis}
\author{Myrto~Kallipoliti}

\address{Department of Mathematics
(Division of Algebra-Geometry)\\
University of Athens\\
Panepistimioupolis\\
15784 Athens, Greece}
\email{caath@math.uoa.gr, mirtok@math.uoa.gr}

\date{June 11, 2007; revised, January 7, 2008.}
\thanks{ 2000 \textit{Mathematics Subject Classification.} Primary
06A11; \, Secondary 05A15, 05E25, 20F55. \\
First author supported by the 70/4/8755 ELKE Research Fund
of the University of Athens.}
\begin{abstract}
The absolute order is a natural partial order on a Coxeter group
$W$. It can be viewed as an analogue of the weak order on $W$ in
which the role of the generating set of simple reflections in $W$
is played by the set of all reflections in $W$. By use of a notion
of constructibility for partially ordered sets, it is proved that
the absolute order on the symmetric group is homotopy
Cohen-Macaulay. This answers in part a question raised by V.
Reiner and the first author. The Euler characteristic of the order
complex of the proper part of the absolute order on the symmetric
group is also computed.
\end{abstract}

\maketitle

\section{Introduction}
\label{intro}

Consider a finite Coxeter group $W$ with set of reflections $T$.
Given $w \in W$, let $\ell_T (w)$ denote the smallest integer $k$
such that $w$ can be written as a product of $k$ reflections in $T$.
The \emph{absolute order}, or \emph{reflection length order}, is
the partial order on $W$ denoted by $\preceq$ and defined by letting
\[ u \preceq v \ \ \ \text{if and only if} \ \ \ \ell_T (u) \, + \,
\ell_T (u^{-1} v) \, = \, \ell_T (v) \]
for $u, v \in W$. Equivalently, $\preceq$ is the partial order on $W$
with cover relations $w \prec wt$, where $w \in W$ and $t \in T$ are
such that $\ell_T (w) < \ell_T (wt)$. We refer to \cite[Section 2.4]{Ar2}
for elementary properties of the absolute order and related
historical remarks and mention that the pair $(W, \preceq)$ is a graded
poset having the identity $e \in W$ as its unique minimal element and
rank function $\ell_T$.

The significance of the absolute order in combinatorics, group
theory, invariant theory and representation theory stems from the
following facts. First, it can be viewed as an analogue of the
weak order \cite[Chapter 3]{BB} on $W$ (this order can be defined
by replacing the generating set of all reflections in $W$, in the
definition of the absolute order, with the set of simple
reflections). Second, the maximal chains in intervals of the form
$[e, w]$ correspond to reduced words of $w$ with respect to the
alphabet $T$ and are relevant in the study of conjugacy classes in
$W$ \cite{Ca}. Third, the rank-generating polynomial of $(W,
\preceq)$ is given by
\[ \sum_{w \in W} \ q^{\ell_T (w)} \ = \ \prod_{i=1}^\ell \
(1 + e_i q), \]
where $e_1, e_2,\dots,e_\ell$ are the exponents \cite[Section 3.20]{Hu}
of $W$ and $\ell$ is its rank. Furthermore, if $c$ denotes a Coxeter
element of $W$, then the combinatorial structure of the intervals in
$(W, \preceq)$ of the form $[e, c]$, known as \emph{noncrossing partition
lattices}, plays an important role in the construction of new monoid
structures and $K(\pi, 1)$ spaces for Artin groups associated with $W$;
see for instance \cite{Be, Br, BrW}.

When $c$ is a Coxeter element, the intervals $[e, c]$ in the
absolute order have pleasant combinatorial and topological
properties. In particular, they were shown to be shellable in
\cite{ABW}. The question of determining the topology of $(W \sm
\{e\}, \preceq)$ was raised by Reiner \cite{Re} \cite[Problem
3.1]{Ar1} and the first author (unpublished) and was also posed by
Wachs \cite[Problem 3.3.7]{Wa}. In this paper we focus on the case
of the symmetric group $\sS_n$ (the case of other Coxeter groups
will be treated in \cite{Ka}). We will denote by $P_n$ the
partially ordered set $(\sS_n, \preceq)$ and by $\bar{P}_n$ its
proper part $(\sS_n \sm \{e\}, \preceq)$. Before we state our main
results, let us describe the poset $P_n$ more explicitly. Given a
cycle $c = (i_1 \ i_2 \ \cdots \ i_r) \in \sS_n$ and indices $1
\le j_1 < j_2 < \cdots < j_s \le r$, we say that the cycle
$(i_{j_1} \ i_{j_2} \ \cdots \ i_{j_s}) \in \sS_n$ can be obtained
from $c$ by deleting elements. Given two disjoint cycles $a, b \in
\sS_n$ each of which can be obtained from $c$ by deleting
elements, we say that $a$ and $b$ are \emph{noncrossing} with
respect to $c$ if there does not exist a cycle $(i \ j \ k \ l)$
of length four which can be obtained from $c$ by deleting
elements, such that $i, k$ are elements of $a$ and $j, l$ are
elements of $b$. For instance, if $n=9$ and $c = (3 \ 5 \ 1 \ 9 \
2 \ 6 \ 4)$ then the cycles $(3 \ 6 \ 4)$ and $(5 \ 9 \ 2)$ are
noncrossing with respect to $c$ but $(3 \ 2 \ 4)$ and $(5 \ 9 \
6)$ are not. It can be checked \cite[Section 2]{Br} that for $u, v
\in \sS_n$ we have $u \preceq v$ if and only if
\begin{itemize}
\itemsep=0pt
\item[$\bullet$]
every cycle in the cycle decomposition for $u$ can be obtained from
some cycle in the cycle decomposition for $v$ by deleting elements
and
\item[$\bullet$]
any two cycles of $u$ which can be obtained from the same cycle $c$
of $v$ by deleting elements are noncrossing with respect to $c$.
\end{itemize}
Figure \ref{fig:abs4} depicts the Hasse diagram of $P_n$ for $n=4$.
We note that the rank of an element $w$ of $P_n$ is equal to $n-p$,
where $p$ is the number of cycles in the cycle decomposition for $w$.
In particular, $P_n$ has rank $n-1$ and its maximal elements are
the cycles in $\sS_n$ of length $n$.

\vspace*{0.1 in}
\begin{figure}
\epsfysize = 3.2 in \centerline{\epsffile{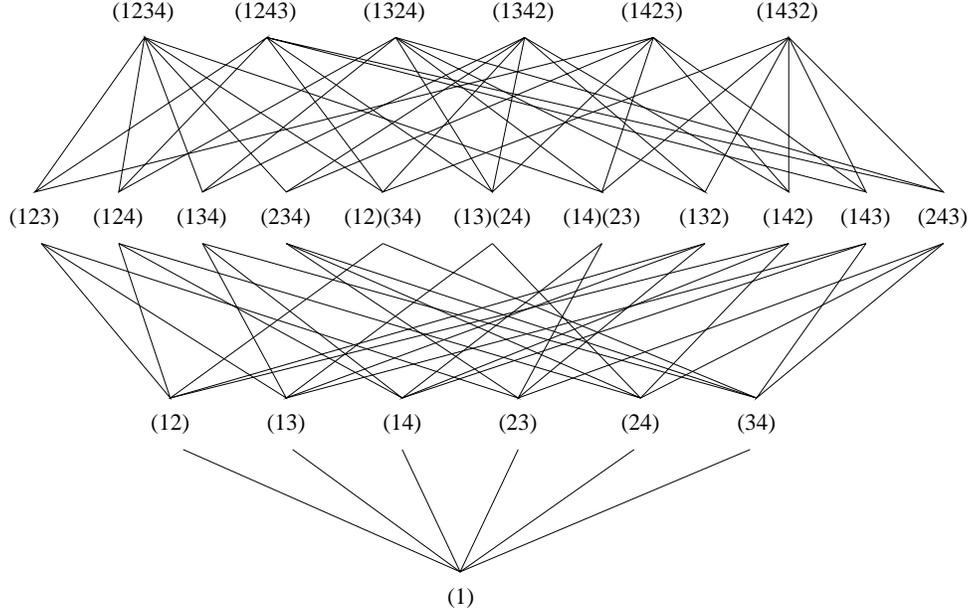}}
\caption{The absolute order on the symmetric group $\sS_4$.}
\label{fig:abs4}
\end{figure}

\medskip
The main results of this paper are as follows.

\begin{theorem}
The poset $\bar{P}_n$ is homotopy Cohen-Macaulay for
all $n \ge 1$. In particular, it is homotopy equivalent to a wedge
of $(n-2)$-dimensional spheres and Cohen-Macaulay over $\ZZ$.
\label{thm1}
\end{theorem}

\begin{theorem}
The reduced Euler characteristic of the order complex $\Delta
(\bar{P}_n)$ satisfies
\begin{equation}
\sum_{n \ge 1} \ (-1)^n \, \tilde{\chi} (\Delta (\bar{P}_n)) \,
\frac{t^n}{n!} \ = \ 1 - C(t) \exp \left\{ -2t \, C(t) \right\},
\label{eq:mobius}
\end{equation}
where $C(t) = \frac{1}{2t} \, (1 - \sqrt{1-4t})$ is the ordinary
generating function for the Catalan numbers.
\label{thm2}
\end{theorem}

Theorems \ref{thm1} and \ref{thm2} are proved in Sections
\ref{proof1} and \ref{proof2}, respectively. Theorem \ref{thm1} is
proved by showing that $P_n$ has a property which we call
strong constructibility. This notion is motivated by the
notion of constructibility for simplicial complexes \cite{Ho} (see
also \cite{St1}) and
is introduced and studied in Section \ref{sec:construct}. Section
\ref{sec:pre} discusses briefly some of the background from
topological combinatorics needed to understand Theorems \ref{thm1}
and \ref{thm2}.

\section{Preliminaries}
\label{sec:pre}

In this section we fix notation, terminology and conventions
related to simplicial complexes and partially ordered sets
(posets) and recall some fundamental definitions and facts. For
more information on these topics we refer the interested reader to
\cite{Bj2}, \cite[Chapter II]{St3}, \cite[Chapter 3]{St2} and
\cite{Wa}. Throughout this paper we use the notation $[n] = \{1,
2,\dots,n\}$.

All simplicial complexes and posets we will consider in this paper
are finite. All topological properties of an abstract simplicial
complex $\Delta$ we mention will refer to those of its geometric
realization $X$ (see \cite[Section 9]{Bj2}). For instance,
$\Delta$ is \emph{$k$-connected} if the homotopy groups $\pi_i (X,
x)$ vanish for all $0 \le i \le k$ and $x \in X$. The elements of
an abstract simplicial complex $\Delta$ are called \emph{faces}.
The \emph{link} of a face $F \in \Delta$ is defined as $\lk_\Delta
(F) = \{G \, \sm F: \, G \in \Delta, \, F \subseteq G\}$. The
complex $\Delta$ is said to be \emph{Cohen-Macaulay} (over $\ZZ$)
if $$\widetilde{H}_i \, (\lk_\Delta (F), \ZZ) = 0$$ for all $F \in
\Delta$ and $i < \dim \lk_\Delta (F)$ and \emph{homotopy
Cohen-Macaulay} if $\lk_\Delta (F)$ is $(\dim \lk_\Delta
(F)-1)$-connected for all $F \in \Delta$. A $d$-dimensional
simplicial complex $\Delta$ is said to be \emph{pure} if all
facets (faces which are maximal with respect to inclusion) of
$\Delta$ have dimension $d$. A pure $d$-dimensional simplicial
complex $\Delta$ is \emph{(pure) shellable} if there exists a
total ordering $G_1, G_2,\dots,G_m$ of the set of facets of
$\Delta$ such that for all $1 < i \le m$, the intersection of $G_1
\cup \, \cdots \, \cup G_{i-1}$ with $G_i$ is pure of dimension
$d-1$. We have the hierarchy of properties

\bigskip
\begin{tabular}{c}
\text{pure shellable} $\Rightarrow$ \text{homotopy Cohen-Macaulay}
$\Rightarrow$ \text{Cohen-Macaulay} $\Rightarrow$ \text{pure}
\end{tabular}

\bigskip
\noindent for a simplicial complex (in Section \ref{sec:construct}
we will insert constructibility between the first and second
property). Moreover, any $d$-dimensional (finite) homotopy
Cohen-Macaulay simplicial complex is $(d-1)$-connected and hence
homotopy equivalent to a wedge of $d$-dimensional spheres.

The order complex, denoted by $\Delta (P)$, of a poset $P$ is the abstract
simplicial complex with vertex set $P$ and faces the chains (totally
ordered subsets) of $P$. All topological properties of a poset $P$ we
mention will refer to those of (the geometric realization of) $\Delta (P)$.
The \emph{rank} of $P$ is defined as the dimension of $\Delta (P)$, in
other words as one less than the largest cardinality of a chain in $P$.
We say that $P$ is \emph{bounded} if it has a minimum and a maximum
element, \emph{graded} if $\Delta (P)$ is pure and \emph{pure shellable}
if so is $\Delta (P)$.
A subset $I$ of $P$ is called an (order) \emph{ideal} if we have
$x \in I$ whenever $x \le y$ holds in $P$ and $y \in I$.

\section{Constructible complexes and posets}
\label{sec:construct}

In this section we introduce the notion of strong constructibility
for partially ordered sets and discuss some of its features which
will be important for us. We will use the following variation of
the notion of constructibility for simplicial complexes \cite{Ho,
St1} \cite[Section 11.2]{Bj2}.
\begin{definition}
A $d$-dimensional simplicial complex $\Delta$ is constructible if
it is a simplex or it can be written as $\Delta = \Delta_1 \cup
\Delta_2$, where $\Delta_1, \Delta_2$ are $d$-dimensional
constructible simplicial complexes such that $\Delta_1 \cap
\Delta_2$ is constructible of dimension at least $d-1$.
\label{def:wconstruct}
\end{definition}
The classical notion of constructibility differs in that, in the
previous definition, the dimension of $\Delta_1 \cap \Delta_2$ has
to equal $d-1$. It is well-known that pure shellability implies
constructibility (in the classical sense). We do not know whether
our notion of constructibility coincides with the (possibly more
restrictive) classical notion. Observe, however, that
constructible simplicial complexes, in the sense of Definition
\ref{def:wconstruct}, are pure and that they enjoy the properties
listed in the following lemma and corollary.
\begin{lemma}
\begin{enumerate}
\itemsep=0pt \item[(i)] If $\Delta$ is a $d$-dimensional
constructible simplicial complex then $\Delta$ is $(d-1)$-connected.
\item[(ii)] If $\Delta$ is constructible then so is the link of
any face of $\Delta$.
\end{enumerate}
\label{lem:wc}
\end{lemma}
\begin{proof}
Part (i) follows from the fact \cite[Lemma 10.3 (ii)]{Bj2} that if
$\Delta_1, \Delta_2$ are $k$-connected and $\Delta_1 \cap \Delta_2$
is $(k-1)$-connected then $\Delta_1 \cup \Delta_2$ is $k$-connected.
Part (ii) follows from the observation that if $F$ is a face of
$\Delta_1 \cup \Delta_2$ then $\lk_{\Delta_1 \cup \Delta_2} (F) =
\lk_{\Delta_1} (F) \cup \lk_{\Delta_2} (F)$ and
$\lk_{\Delta_1} (F) \cap \lk_{\Delta_2} (F) = \lk_{\Delta_1 \cap
\Delta_2} (F)$.
\end{proof}
\begin{corollary}
If $\Delta$ is a constructible simplicial complex then $\Delta$ is
homotopy Cohen-Macaulay. \label{cor:wCM}
\end{corollary}
\begin{proof}
This follows from Lemma \ref{lem:wc}.
\end{proof}
\begin{lemma}
Let $\Delta_1, \Delta_2,\dots,\Delta_k$ be $d$-dimensional
constructible simplicial complexes.
\begin{enumerate}
\itemsep=0pt \item[(i)] If the intersection of any two or more of
$\Delta_1, \Delta_2,\dots,\Delta_k$ is constructible of
dimension $d$, then their union is also constructible.
\item[(ii)] If the intersection of any two or more of $\Delta_1,
\Delta_2,\dots,\Delta_k$ is constructible of dimension
$d-1$, then their union is also constructible.
\end{enumerate}
\label{lem:wconstruct}
\end{lemma}
\begin{proof}
We proceed by induction on $k$. The case $k=1$ is trivial and the
case $k=2$ is clear by definition, so we assume that $k \ge 3$.
The complexes $\Delta_1 \cap \Delta_k,\dots,\Delta_{k-1} \cap
\Delta_k$ have dimension $d$ or $d-1$ in the cases of parts (i)
and (ii), respectively, and satisfy the hypothesis of part (i).
Hence, by our induction hypothesis, their union $(\Delta_1 \cup
\cdots \cup \Delta_{k-1}) \cap \Delta_k$ is constructible of
dimension $d$ or $d-1$, respectively. Since, by induction,
$\Delta_1 \cup \cdots \cup \Delta_{k-1}$ is constructible of
dimension $d$ and, by assumption, so is $\Delta_k$, it follows
that $\Delta_1 \cup \cdots \cup \Delta_k$ is constructible as
well.
\end{proof}

We now consider the class of finite posets with a minimum element
and define the notion of strong constructibility as follows.
\begin{definition}
A finite poset $P$ of rank $d$ with a minimum element is strongly
constructible if it is bounded and pure shellable or it can be
written as a union $P = I_1 \cup I_2$ of two strongly
constructible proper ideals $I_1, I_2$ of rank $d$, such that $I_1
\cap I_2$ is strongly constructible of rank at least $d-1$.
\label{def:wconstruct2}
\end{definition}

Note that any strongly constructible poset is graded.
\begin{proposition}
The order complex of any strongly constructible poset is
constructible. \label{prop:wc}
\end{proposition}
\begin{proof}
Let $P$ be a strongly constructible poset of rank $d$. To show
that $\Delta (P)$ is constructible we will use induction on the
cardinality of $P$. If $P$ is pure shellable then $\Delta (P)$ is
pure shellable and hence constructible. Otherwise $P$ is the union
of two strongly constructible proper ideals $I_1, I_2$ of rank
$d$, such that $I_1 \cap I_2$ is strongly constructible of rank at
least $d-1$. Clearly we have $\Delta (P) = \Delta (I_1) \cup
\Delta (I_2)$ and $\Delta (I_1) \cap \Delta (I_2) = \Delta (I_1
\cap I_2)$. Since, by the induction hypothesis, $\Delta (I_1)$ and
$\Delta (I_2)$ are constructible of dimension $d$ and $\Delta (I_1
\cap I_2)$ is constructible of dimension at least $d-1$, it
follows that $\Delta (P)$ is constructible as well. This completes
the induction and the proof of the proposition.
\end{proof}

The next lemma asserts that our notion of strong constructibility
for posets behaves well under direct products.
\begin{lemma}
The direct product of two strongly constructible posets is
strongly constructible. \label{lem:product}
\end{lemma}
\begin{proof}
Let $P, Q$ be two strongly constructible posets of ranks $d$ and
$e$, respectively. We proceed by induction on the sum of the
cardinalities of $P$ and $Q$. If $P$ and $Q$ are both bounded and
pure shellable then their direct product $P \times Q$ is also
(bounded and) pure shellable \cite[Theorem 8.3]{BW} and hence
strongly constructible. If not then one of them, say $P$, can be
written as a union $P = I_1 \cup I_2$ of two strongly
constructible proper ideals $I_1, I_2$ of rank $d$, such that $I_1
\cap I_2$ is strongly constructible of rank at least $d-1$. Then
$P \times Q$ is the union of its proper ideals $I_1 \times Q$ and
$I_2 \times Q$, each of rank $d+e$. By our induction hypothesis,
these products are strongly constructible and so is their
intersection $(I_1 \cap I_2) \times Q$, which has rank at least
$d+e-1$. As a result, $P \times Q$ is strongly constructible as
well.
\end{proof}

The proof of the following lemma is analogous to that of Lemma
\ref{lem:wconstruct} (ii) and is omitted.
\begin{lemma}
Let $P$ be a finite poset of rank $d$ with a minimum element. If
$P$ is the union of strongly constructible ideals $I_1,
I_2,\dots,I_k$ of $P$ of rank $d$ and the intersection of any two
or more of these ideals is  strongly constructible of rank $d-1$,
then $P$ is strongly constructible. \qed \label{lem:wconstruct2}
\end{lemma}

\section{Proof of Theorem \ref{thm1}}
\label{proof1}

In this section we prove Theorem \ref{thm1} by showing that $P_n$
is strongly constructible. We will in fact prove a more general
statement. For that reason, we introduce the following notation.
Let $\tau_0, \tau_1,\dots,\tau_k$ be pairwise disjoint subsets of
$[n]$, such that $\tau_1,\dots,\tau_k$ are nonempty. Let also
$\sigma$ be a nonempty sequence of distinct elements of $[n]$,
none of which belongs to any of the sets $\tau_i$. We set $R =
(\sigma, \tau_0,\dots,\tau_k)$ and denote by $\sS_n (R)$ the set
of permutations $w \in \sS_n$ which have exactly $k+1$ cycles
$c_0, c_1,\dots,c_k$ in their cycle decomposition, such that
\begin{enumerate}
\itemsep=0pt
\item[(a)]
the elements of $\sigma$ appear consecutively in the cycle $c_0$ in the
order in which they appear in $\sigma$ and
\item[(b)]
the elements of $\tau_i$ appear in the cycle $c_i$ for $0 \le i \le k$.
\end{enumerate}
\begin{example}
{\rm Suppose $k=0$ and $\sigma = (1, 2,\dots,r)$. Then $\sS_n (R)$ is
the set of cycles $w \in \sS_n$ of length $n$ for which $w(i) = i+1$
for $1 \le i \le r-1$. In particular, if $r=1$ then $\sS_n (R)$ is the
set of all maximal elements of $P_n$.}
\label{ex:k=0}
\end{example}

The following proposition is the main result in this section.
\begin{proposition}
If $R$ is as above then the order ideal of $P_n$ generated by
$\sS_n (R)$ is strongly constructible. \label{prop:gen}
\end{proposition}

The next remark will be used in the proof of the following
technical lemma, which will be used in turn in the proof of
Proposition \ref{prop:gen}.
\begin{remark}
{\rm Suppose that $w \preceq c$ holds in $\sS_n$, where $c = (a_1
\ a_2 \ \cdots \ a_n)$ is a cycle of length $n$, and let $1 \le p
\le n$. Suppose further that $w$ has a cycle containing no $a_i$
with $1 \le i \le p$. Then there exists a permutation in $\sS_n$
which has exactly two cycles $u, v$ in its cycle decomposition,
such that $u(a_i) = a_{i+1}$ for $1 \le i \le p-1$, the elements
appearing in $u$ are exactly the elements which appear in those
cycles of $w$ containing $a_1, a_2,\dots,a_p$, and $w \preceq uv$.
This statement follows easily from the description of the absolute
order on $\sS_n$ given in Section \ref{intro}. One constructs the
cycle $u$ by merging appropriately the cycles of $w$ in which the
elements $a_1, a_2,\dots,a_p$ appear. The cycle $v$ can be
constructed by merging appropriately the remaining cycles of $w$.
The details are left to the reader. \qed} \label{lem:tech0}
\end{remark}
\begin{lemma}
Let $1 \le r \le n-2$ and $J$ be a subset of $\{r+1,\dots,n\}$ with
at least two elements. Suppose that
$w \in \sS_n$ is such that for all $j \in J$ there exists a cycle
$c$ in $\sS_n$ of length $n$ satisfying $c(i) = i+1$ for $1 \le i
\le r-1$, $c(r) = j$ and $w \preceq c$.
\begin{enumerate}
\itemsep=0pt
\item[(i)] There exists at most one $j \in J$ such that $i$ and $j$
are elements of the same cycle in the cycle decomposition for $w$ for
some $1 \le i \le r$.
\item[(ii)] There exists a permutation in $\sS_n$ which has exactly two
cycles $u, v$ in its cycle decomposition such that $w \preceq uv$, $u(i)
= i+1$ for $1 \le i \le r-1$ and one of the following holds: {\rm (a)}
all elements of $J$ appear in $v$, or {\rm (b)} there exists $j \in J$
with $u(r) = j$ and all other elements of $J$ appear in $v$.
\end{enumerate}
\label{lem:tech}
\end{lemma}
\begin{proof}
Part (i) is once again an easy consequence of the description of
the absolute order on $\sS_n$ given in Section \ref{intro}. Part
(ii) follows from part (i) and Remark \ref{lem:tech0} (the latter
is applied either for $p = r$ to $w$ and a cycle $c$ of length $n$
satisfying $c(i) = i+1$ for $1 \le i \le r-1$, if no element of
$J$ appears in the same cycle of $w$ with some $1 \le i \le r$, or
for $p=r+1$ and a cycle $c$ of length $n$ satisfying $c(i) = i+1$
for $1 \le i \le r-1$ and $c(r) = j$, if $j \in J$ appears in the
same cycle of $w$ with some $1 \le i \le r$).
\end{proof}

\medskip
\noindent \emph{Proof of Proposition \ref{prop:gen}.} We denote by
$\iI_n (R)$ the order ideal of $P_n$ generated by $\sS_n (R)$, so
that $\iI_n (R)$ is a graded poset of rank $n-k-1$, and by $m$ the
number of elements of $[n]$ not appearing in $R = (\sigma,
\tau_0,\dots,\tau_k)$. We proceed by induction on $n$, $n-k$ and
$m$, in this order.

We assume $n \ge 3$, the result being trivial otherwise. We first
treat the case $k \ge 1$. For $m=0$, the poset $\iI_n (R)$ is
isomorphic to the direct product $\iI_r (S) \times P_{r_1} \times
\cdots \times P_{r_k}$ where $S = (\sigma, \tau_0)$, $r$ is the
number of elements of $[n]$ appearing in $S$ and $r_i$ is the
cardinality of $\tau_i$ for $1 \le i \le k$. Since $\iI_r (S)$ and
$P_{r_i}$ are strongly constructible by our induction hypothesis
on $n$, the poset $\iI_n (R)$ is strongly constructible by Lemma
\ref{lem:product}. Suppose now that $m \ge 1$ and let $j$ be an
element of $[n]$ which does not appear in $R$. Clearly we have
\[ \iI_n (R) \ = \ \bigcup_{i=0}^k \, \iI_n (R_i), \]
where $R_i$ is obtained from $R$ by adding $j$ in the set
$\tau_i$. Each ideal $\iI_n (R_i)$ has rank $n-k-1$ and, by our
induction hypothesis on $m$, it is strongly constructible.
Moreover, the intersection of any two or more of these ideals is
equal to $\iI_n (S)$, where $S = (\sigma,
\tau_0,\dots,\tau_{k+1})$ with $\tau_{k+1} = \{j\}$. Since $\iI_n
(S)$ has rank $n-k-2$, it is strongly constructible by our
induction hypothesis on $n-k$. It follows from Lemma
\ref{lem:wconstruct2} that $\iI_n (R)$ is strongly constructible
as well.

Finally, suppose that $k=0$. Since the elements of $\tau_0$ are
irrelevant in this case, we may assume that $\tau_0$ is empty.
Clearly, the isomorphism type of $\iI_n (R)$ depends only on the
length of $\sigma$. Thus, for convenience with the notation, we
will also assume that $\sigma = (1, 2,\dots,r)$ for some $1 \le r
\le n$. If $m=0$, so that $r=n$, then $\sS_n (R)$ consists of a
single cycle of length $n$ and $\iI_n (R)$ is isomorphic to the
lattice of noncrossing partitions of $[n]$. Thus $\iI_n (R)$ is
bounded and pure shellable \cite[Example 2.9]{Bj1} and, in
particular, strongly constructible. Suppose that $m \ge 1$, so
that $r \le n-1$. For $r+1 \le j \le n$ we set $R_j = (\sigma_j,
\varnothing)$, where $\sigma_j = (1,\dots,r, j)$ is obtained from
$\sigma$ by attaching $j$ at the end. Each ideal $\iI_n (R_j)$ has
rank $n-1$ and, by our induction hypothesis on $m$, it is strongly
constructible. Moreover, we have
\[ \iI_n (R) \ = \ \bigcup_{j=r+1}^n \, \iI_n (R_j). \]
In view of Lemma \ref{lem:wconstruct2}, to prove that $\iI_n (R)$
is strongly constructible it suffices to show that the
intersection of any two or more of the ideals $\iI_n (R_j)$ is
strongly constructible of rank $n-2$. Let $J$ be any subset of
$\{r+1,\dots,n\}$ with at least two elements. We claim that
\begin{equation}
\bigcap_{j \in J} \, \iI_n (R_j) \ = \ \iI_n (S_0) \, \cup \,
\left( \bigcup_{j \in J} \, \iI_n (S_j) \right), \label{eq:inter1}
\end{equation}
where $S_0 = (\sigma, \varnothing, J)$ and $S_j = (\sigma_j,
\varnothing, J \sm \{j\})$ for $j \in J$. Indeed, it should be
clear that each ideal $\iI_n (S_j)$ for $j \in J \cup \{0\}$ is
contained in the intersection in the left hand-side of
(\ref{eq:inter1}). The reverse inclusion follows from Lemma
\ref{lem:tech} (ii). Next, we note that the ideals $\iI_n (S_j)$
for $j \in J \cup \{0\}$ have rank $n-2$ and that, by our
induction hypothesis on $n-k$, they are strongly constructible.
Applying induction on the cardinality of $J$, to show that the
union in the right hand-side of (\ref{eq:inter1}) is strongly
constructible it suffices to show that for $q \in J$, the
intersection
\[ \iI_n (S_q) \, \cap \, \left( \bigcup_{j \in (J \sm \{q\})
\cup \{0\}} \, \iI_n (S_j) \right) \]
is strongly constructible of rank $n-3$. We claim that this
intersection is equal to $\iI_n (S)$, where $S = (\sigma,
\varnothing, J \sm\{q\}, \{q\})$. Indeed, one inclusion follows
from the fact that $\iI_n (S) \subseteq \iI_n (S_q) \cap \iI_n
(S_0)$. For the reverse inclusion observe that in each permutation
in $\sS_n (S_q)$, $q$ appears in a cycle containing $1, 2,\dots,r$
but no element of $J \sm \{q\}$ and that for all $j \in (J \sm
\{q\}) \cup \{0\}$, in each permutation in $\sS_n (S_j)$, $q$
appears in a cycle containing none of $1, 2,\dots,r$. Finally,
observe that the ideal $\iI_n (S)$ has the desired rank $n-3$ and
is strongly constructible by our induction hypothesis on $n-k$.
This completes the induction and the proof of the proposition.
\qed

\medskip
\noindent \emph{Proof of Theorem \ref{thm1}.} When $k=0$ and
$\sigma$ has length one (see Example \ref{ex:k=0}), the ideal
$\iI_n (R)$ coincides with $P_n$. Therefore Proposition
\ref{prop:gen} implies that $P_n$ is strongly constructible. It
follows from Proposition \ref{prop:wc} and Corollary \ref{cor:wCM}
that $P_n$ is homotopy Cohen-Macaulay. As a result, so is 
$\bar{P}_n$. \qed

\section{Proof of Theorem \ref{thm2}}
\label{proof2}

In this section we denote by $\hat{0}$ the minimum element of $P_n$
and by $\hat{P}$ the poset obtained from $P_n$ by adding a maximum
element $\hat{1}$.

\medskip
\noindent \emph{Proof of Theorem \ref{thm2}.} From \cite[Proposition
3.8.6]{St2} we have that $\tilde{\chi} (\Delta (\bar{P}_n)) = \mu_n (\hat{0},
\hat{1})$, where $\mu_n$ stands for the M\"obius function of $\hat{P}_n$,
and hence that
\begin{equation}
\tilde{\chi} (\Delta (\bar{P}_n)) \ = \ - \sum_{x \in P_n} \ \mu_n
(\hat{0}, x).
\label{eq:mu1}
\end{equation}
Let $C_m = \frac{1}{m+1} {2m \choose m}$ denote the $m$th Catalan number.
It is well known (see, for instance, \cite[Exercise 3.68 (b)]{St2}) that
$$\mu_n ( \hat{0}, x) = (-1)^{k-1} C_{k-1}$$ if $x \in \sS_n$ is a cycle of
length $k$, since in this case the interval $[\hat{0}, x]$ in $P_n$ is
isomorphic to the lattice of noncrossing partitions of the set $[k]$.
Moreover, for any $x \in \sS_n$ the interval $[\hat{0}, x]$ is isomorphic
to the direct product over the cycles $y$ in the cycle decomposition for
$x$ of the intervals $[\hat{0}, y]$. Therefore we have
\begin{equation}
\mu_n (\hat{0}, x) \ = \ \prod_{y \in \cC(x)} \ (-1)^{\# y - 1} \,
C_{\# y - 1},
\label{eq:mu2}
\end{equation}
where $\cC(x)$ is the set of cycles in the cycle decomposition for $x$
and $\# y$ is the cardinality of $y$.
Given (\ref{eq:mu1}) and (\ref{eq:mu2}), the exponential formula
\cite[Corollary 5.1.9]{St4} implies that
\begin{equation}
1 - \sum_{n \ge 1} \ \tilde{\chi} (\Delta (\bar{P}_n)) \, \frac{t^n}{n!}
\ = \ \exp \ \sum_{n \ge 1} \ (-1)^{n-1} C_{n-1} \, \frac{t^n}{n}.
\label{eq:mu3}
\end{equation}
Integrating the well known equality
\[ \sum_{n \ge 1} \ C_{n-1} \, t^{n-1} \ = \ \frac{1 - \sqrt{1-4t}}{2t} \]
we get
\[ \sum_{n \ge 1} \ C_{n-1} \, \frac{t^n}{n} \ = \ 1 - \sqrt{1-4t} \, + \,
\log \left( 1 + \sqrt{1-4t} \right) - \log 2. \]
Switching $t$ to $-t$ in the previous equality and exponentiating, we get
\[ \exp \ \sum_{n \ge 1} \ (-1)^{n-1} C_{n-1} \frac{t^n}{n} \ = \
\frac{\sqrt{1+4t}-1}{2t} \ \exp \left( \sqrt{1+4t}-1 \right). \]
In view of the previous equality, the result follows by switching $t$ to
$-t$ in (\ref{eq:mu3}).
\qed

\begin{remark}
{\rm It follows from Theorems \ref{thm1} and \ref{thm2} that if
$C(t) = \frac{1}{2t} \, (1 - \sqrt{1-4t})$ then the generating
function $$1 - C(t) \exp \left\{ -2t \, C(t) \right\}$$ has
nonnegative coefficients.} \qed
\end{remark}
Table \ref{numbers} lists the first few values of $(-1)^n
\tilde{\chi} (\Delta (\bar{P}_n))$.

\begin{table}[hptb]
\begin{center}
\begin{tabular}{| c| c | c | c | c | c | c | c | c | c | c |} \hline
 \ \ \ \ $n$ & 1 & 2 & 3 & 4 & 5 & 6 & 7 & 8 & 9\\ \hline
 $(-1)^n \tilde{\chi} (\Delta (\bar{P}_n))$
& 1 & 0 & 2 & 16 & 192 & 3008 & 58480 & 1360896 & 36931328
\\ \hline
\end{tabular}
\caption{The numbers $(-1)^n \, \tilde{\chi} (\Delta (\bar{P}_n))$
for $n \le 9$. }
 \label{numbers}
\end{center}
\end{table}

\end{document}